\documentclass[12pt,a4paper,reqno]{amsart}
\usepackage{amssymb,amsmath}
\usepackage{amsfonts,amsbsy,bm}
\usepackage[dvipsnames]{xcolor}
\usepackage{latexsym}
\usepackage{exscale}
\usepackage{enumerate}
\usepackage{soul}
\usepackage{amsthm}
\usepackage[colorlinks, bookmarks=true]{hyperref}
\usepackage{hyperref}

\usepackage{color}

\newcommand{\N}{\mathbb N}

\newcommand{\D}{\mathcal D}

\newcommand{\be}{{\mathbf e}}

\newcommand {\X} {{\mathbb X}}
\newcommand {\Y} {{\mathbb Y}}

\newcommand {\e} {{\varepsilon}}

\def\d{\displaystyle}

\newcommand{\bfe}{{\boldsymbol\e}}

\renewcommand{\phi}{{\varphi}}

\def\supp{\mathop{\rm supp}}
\def\sgn{\mathop{\rm sign}}

\numberwithin{equation}{section}
\newtheorem*{theorem*}{Theorem}
\newtheorem{theorem}{Theorem}[section]
\newtheorem{lemma}[theorem]{Lemma}

\newtheorem{corollary}[theorem]{Corollary}
\newtheorem{Remark}[theorem]{Remark}
\newtheorem{remark}[theorem]{Remark}

\newtheorem{proposition}[theorem]{Proposition}
\newtheorem{definition}[theorem]{Definition}
\newtheorem{example}[theorem]{Example}

\newcommand{\Ba}[1]{\begin{array}{#1}}
\newcommand{\Ea}{\end{array}}
\newcommand{\Be}{\begin{equation}}
\newcommand{\Ee}{\end{equation}}
\newcommand{\Bea}{\begin{eqnarray}}
\newcommand{\Eea}{\end{eqnarray}}
\newcommand{\Beas}{\begin{eqnarray*}}
\newcommand{\Eeas}{\end{eqnarray*}}
\newcommand{\Benu}{\begin{enumerate}}
\newcommand{\Eenu}{\end{enumerate}}
\newcommand{\Bi}{\begin{itemize}}
\newcommand{\Ei}{\end{itemize}}

\newcommand{\BR}{\begin{Remark} \em}
\newcommand{\ER}{\end{Remark}}
\newcommand{\BE}{\begin{example} \em}
\newcommand{\EE}{\end{example}}

\def\txcr{\textcolor{red}}

\def\co{\text{\rm co}}

\newcommand {\bone} {{\bf 1}}

\newcounter{reg}
\newcounter{regTO}
\newcommand{\vertiii}[1]{{\left\vert\kern-0.25ex\left\vert\kern-0.25ex\left\vert #1
		\right\vert\kern-0.25ex\right\vert\kern-0.25ex\right\vert}}

\newcommand{\bff}{\mathbf 1}
\newcommand{\cupdot}{\mathbin{\mathaccent\cdot\cup}}
\renewcommand\Re{\operatorname{Re}}
\renewcommand\Im{\operatorname{Im}}

\begin{document}
\title[strong partially greedy bases and Lebesgue-type inequalities]{strong partially greedy bases and Lebesgue-type inequalities}

\author{M. Berasategui}

\address{Miguel Berasategui
	\\
 IMAS - UBA - CONICET - Pab I, Facultad de Ciencias Exactas y Naturales \\ Universidad de Buenos Aires \\ (1428), Buenos Aires, Argentina}
	 \email{mberasategui@dm.uba.ar}

\author{P.M. Bern\'a}

\address{Pablo M. Bern\'a
	\\
	Departmento de Matem\'aticas
	\\
	Universidad Aut\'onoma de Madrid
	\\
	28049 Madrid, Spain} \email{pablo.berna@uam.es}
	 
\author{S. Lassalle}
\address{Silvia Lassalle
\\Departamento de Matem\'atica \\
Universidad de San Andr\'es, Vito Duma 284\\
(1644) Victoria, Buenos Aires, Argentina and \\
IMAS - CONICET
} \email{slassalle@udesa.edu.ar}

\begin{abstract}
In this paper we continue the study of Lebesgue-type inequalities for greedy algorithms. We introduce  the notion of  strong partially greedy Markushevich bases and study the Lebesgue-type parameters associated with them. We prove that this property is equivalent to that of being conservative and quasi-greedy, extending a similar result given in \cite{DKKT} for Schauder bases. 
We also give a characterization of 1-strong partial greediness, following the study started in \cite{AW, AA1}.
\end{abstract}

\thanks{The first and third authors were supported in part by CONICET PIP 11220130100483 and ANPCyT PICT-2015-2299. The second author was partially supported by the grants MTM-2016-76566-P (MINECO, Spain) and 20906/PI/18 from Fundaci\'on S\'eneca (Regi\'on de Murcia, Spain), and has received funding from the European Union's Horizon 2020 research and innovation programme under the Marie Sk\l odowska-Curie grant agreement No. 777822. The third author was also supported by  PAI-UdeSA 2019. 
}

\date{\today}
\subjclass[2010]{41A65, 41A46, 41A17, 46B15, 46B45.}

\keywords{Non-linear approximation, Lebesgue-type inequality, greedy algorithm, quasi-greedy basis, partially greedy bases.}

\maketitle
\section{Introduction and background}

Throughout this paper  $\X$ is a separable infinite dimensional Banach space over the field  $\mathbb F=\mathbb R$ or $\mathbb C$, with a \textit{semi-normalized Markushevich basis}  $\mathcal B=(\be_n)_{n=1}^\infty$ (M-basis for short). That is, denoting with $\X^*$ the dual space of $\X$, $(\be_n)_{n=1}^\infty$ satisfies the following conditions:
\begin{enumerate}[\upshape (i)]
	\item $\X=\overline{[\be_n \colon n\in\N]}$,
	\item there is a (unique) sequence $(\be_n^*)_{n=1}^\infty\subset \X^*$, called  biorthogonal functionals, such that $\be_k^*(\be_n)=\delta_{k,n}$ for all $k,n\in\N$.
	\item  if $\be_n^*(x)=0$ for all $n$, then $x=0$.
	\item There exist a constant $c>0$ such that $\sup_n\lbrace\Vert \be_n\Vert, \Vert \be_n^*\Vert\rbrace\leq c$.
\end{enumerate}

Clearly, the semi-normalized condition of the basis is guaranteed by (ii) and (iv). Under the above  assumptions, every $x\in\X$ is associated with  a formal series $x \sim \sum_{n=1}^\infty \be_n^*(x) \be_n$, so that $\lim_n \be_n^*(x)=0$ and its coefficients $(\be_n^*(x))_{n=1}^\infty$ are uniquely determined. As usual, we denote by $\supp{(x)}$ the support of $x\in \X$, that is the set $\{n\in \mathbb{N}: e_n^*(x)\not=0\}$. For $A$ and $B$ subsets of $\mathbb N$, we write $A<B$ to mean that  $\max A<\min B$. If $m\in \mathbb N$, we write $m < A$ and $A < m$ for $\{ m\} < A$ and  
$A <\{ m\}$ respectively. Also, $A\cupdot B$ means the union of $A$ and $B$ with $A\cap B=\emptyset$. 

We recall a few standard notions about greedy algorithms.  For further information on the subject we refer the reader to the monograph by V. N. Temlyakov \cite{T0}  and to the articles \cite{DKKT, KT, T, Wo} and the reference therin. For Lebesgue-type inequalities see, for instance, the more recent papers \cite{BBG, BBGHO2, DKO2015, DST, GHO2013}). 

Given $x\in\mathbb X$, a \textbf{greedy set} for $x\in \X$ of order $m$ (or an $m$-greedy set for $x$) is a set of indices $A\subset\N$ such that $|A|=m$ and
$$
\min_{n\in A} |\be_n^*(x)| \geq \max_{n\notin A} |\be_n^*(x)|\,.
$$
A {\bf greedy operator} of order $m$ is any mapping $G_m\colon \X\to\X$ such that
$$
x\in\X
\longmapsto  G_m(x) =G_m[\mathcal B,\X]:=
\sum_{n\in A} \be_n^*(x) \be_{n}\,
$$
with $A$ an $m$-greedy set for $x$, using the convention that  $G_0 =0$. We write $\mathcal G_m$ for the set of all greedy operators of order $m$ and we consider $\mathcal G =\cup_{m\ge 0} \mathcal G_m$.  Given  $G$ and $G'$ in $\mathcal G$ we write $G' < G$ whenever $G'\in \mathcal G_m$ and $G\in \mathcal G_n$ with $0\le m< n$ and the respective supporting sets satisfy $\supp(G')\subset\supp(G)$ (for all $x$). 

Given a finite set $A\subset\N$, we denote by $P_A(x)=\sum_{n\in A}\be_n^*(x)\be_n$ the projection operator, with the convention that any sum over the empty set is zero.

For $A\subset\N$ finite, we denote by $\Psi_A$ the set of all collections of sequences $\bfe = (\e_n)_{n\in A}\subset \mathbb F$ such that $\vert \varepsilon_n\vert=1$ and
$$
\bff_{\bfe A}=\bff_{\bfe A}[\mathcal B,\X]:=\sum_{n\in A}\e_n \be_n.
$$
If $\bfe\equiv 1$, we just write $\bff_A$. Also, every time we have index sets $A\subset B$ and $\bfe\in \Psi_B$, we write $\bff_{\bfe A}$ considering the  natural restriction of $\bfe$ to $A$.

Given $x\in\X$, the error in the best \emph{$m$-term approximation} with respect to the basis $\mathcal B$ is 
$$
\sigma_m(x)=\sigma_m[\mathcal B,\X](x):=\inf\left\lbrace\left\Vert x-y\right\Vert\colon |\supp{(y)}|\le m\right\rbrace,
$$
and the error in the best \emph{$m$-term coordinate approximation} is
$$
\widetilde{\sigma}_m(x)=\widetilde\sigma_m[\mathcal B,\X](x):=\inf\left\lbrace\left\Vert x-P_B(x)\right\Vert\colon  \vert B\vert\leq m\right\rbrace.
$$

Greedy operators are frequently used for $m$-term approximations and, in order to study and quantify the performance of greedy operators one considers, for every $m=1,2, \dots$, the smallest numbers $\mathbf L_m=\mathbf L_m[\mathcal B,\X]$ and $\widetilde{\mathbf{L}}_m=\widetilde{\mathbf{L}}_m[\mathcal B,\X]$ such that for all  $x\in\mathbb X$ and all $G_m\in \mathcal G_m$,
\begin{eqnarray}\label{greedy}
\Vert x-G_m(x)\Vert \leq \mathbf L_m\sigma_m(x),
\end{eqnarray}
and
\begin{eqnarray}\label{agreedy}
\Vert x-G_m(x)\Vert \leq \widetilde{\mathbf L}_m\widetilde\sigma_m(x).
\end{eqnarray}

The parameters $\mathbf{L}_m$ and $\widetilde{\mathbf L}_m$ are called Lebesgue-type parameters and \eqref{greedy} and \eqref{agreedy} their respective  Lebesgue-type inequalities. When $\mathbf L_m = O(1)$, the basis is called greedy \cite{KT} and if $\widetilde{\mathbf L}_m = O(1)$, the basis is called almost-greedy  \cite{DKKT}. Thanks to the main results of  \cite{DKKT} and \cite{KT}, we know that a basis $\mathcal B$ is greedy if and only if $\mathcal B$ is unconditional and democratic, and $\mathcal B$ is almost-greedy if and only if it is quasi-greedy (see Definition~\ref{definitionquasigreedy} below) and democratic. Recall that a basis \txcr{$\mathcal B$} is $K$-\textit{unconditional}, with $K>0$, if
$$
K=K[\mathcal B,\X]:=\sup_{\vert A\vert<\infty}\Vert P_A\Vert<\infty,
$$
and $\mathcal B$ is $D$-\textit{democratic} with $D>0$, if for any pair of finite sets $A, B$ with $|A| \le |B|$ we have 
$$
\Vert \bff_A\Vert \leq D\Vert \bff_B\Vert.
$$

Different estimates for the parameters $\mathbf L_m$ and $\widetilde{\mathbf L}_m$ have been studied in several papers: for example, in \cite{DKO2015, DST, GHO2013} the authors study estimates under the assumption of quasi-greediness, and in \cite{BBG, BBGHO}, for general Markushevich bases. 

In  order to study whether greedy approximations $(G_m(x))_m$ always perform better than the standard linear approximation $(P_m(x))_m$, the  \textit{$m$th residual Lebesgue constant} was introduced in \cite{DKO2015} as follows: $\mathbf L_m^{re}=\mathbf L_m^{re}[\mathcal B, \X]$ is the smallest number such that for all $x\in\X$  and $G_m\in\mathcal G_m$,
$$
\Vert x-G_m(x)\Vert \leq \mathbf L_m^{re}\Vert x-P_m(x)\Vert,
$$
where $P_m$ is the $m$th partial sum, that is, $P_m(x)=P_{A}(x)$ for $A=\lbrace 1,\ldots,m\rbrace$.

If $C_p:=\sup_m \mathbf L_m^{re}<\infty$, the basis is called $C_p$-\textit{partially greedy}. Here, following the spirit of the inequalities \eqref{greedy} and \eqref{agreedy}, and using the definition of $w$-partially greedy bases given in \cite{BDKOW}, we introduce the strong residual Lebesgue-type parameter as follows: given $x\in\X$ and $m\in\N$, we define the \textit{$m$th strong residual error} as
$$
\widehat{\sigma}_m(x)=\widehat\sigma_m[\mathcal B,\X](x):=\inf_{k\leq m}\Vert x-P_k(x)\Vert.
$$
Then, for each $m=1,2,\dots,$ we define the \textit{strong residual Lebesgue-type parameter} as the smallest number $\widehat{\mathbf L}_m=\widehat{\mathbf L}_m[\mathcal B,\X]$ such that for all $x\in\X$  and $G_m\in\mathcal G_m$,
$$
\Vert x-G_m(x)\Vert \leq \widehat{\mathbf L}_m\widehat{\sigma}_m(x).
$$
\vskip .2cm

\begin{definition}
An M-basis $\mathcal B$ is $C_{sp}$-\textbf{strong partially greedy} if
$$
C_{sp}:=\sup_m\widehat{\mathbf L}_m<\infty.
$$
\end{definition}

Clearly, $\mathbf L_m^{re}\leq \widehat{\mathbf L}_m$ and, if the basis is Schauder, $\widehat{\mathbf L}_m \approx \mathbf L_m^{re}$. Consequently, if $\mathcal B$ is Schauder, the basis is partially greedy if and only if it is strong partially greedy. For quasi-greedy bases, different bounds for $\mathbf L_m^{re}$ were studied in \cite{DKO2015}. 

\begin{definition}[\cite{Wo}]\label{definitionquasigreedy}
An M-basis $\mathcal B$ is \textbf{quasi-greedy} if $C_q:=\sup_m g_m^c<\infty$, where~\footnote{We use the notation $\|G\|=\sup_{x\not=0}\|G(x)\|/\|x\|$ and $\|I-G\|=\sup_{x\not=0}\|x-G(x)\|/\|x\|$, even if $G:\mathbb X\to\mathbb X$ is a non-linear map.} 
$$
g_m^c:=\sup_{G\in {\bigcup}_{k\leq m} \mathcal G_k}\Vert I-G\Vert.
$$
\end{definition}
Related to the parameter $g_m^c$, we also consider (see \cite{BBG}) 
$$
g_m:=\sup_{G\in \bigcup_{k\leq m}\mathcal G_k}\Vert G\Vert\qquad \text{and}\qquad\  \tilde{g}_m:= \sup_{\substack{G\in\bigcup_{k\leq m}\mathcal G_k
 \\  G'<G}} \Vert G-G' \Vert. 
$$

Notice that as $\mathcal G_0=\{0\}$, we may define $g_m^c$ (and also $g_m$) for $m=0$ being $g_0^c=1$ (and $g_0=0$). This is the only parameter we will use with $m=0$. On the other hand, in the definitions of the parameters that follow, we avoid without specifying it the undesirable situations in which the denominator could be zero.

\begin{definition}[\cite{BBG, Wo}]\label{C_u:unconditionall for cte coeff}
An M-basis $\mathcal B$ is $C_u$-\textbf{unconditional for constant coefficients} if $C_u:=\sup_m \gamma_m<\infty$, where
$$\gamma_m := \sup\left\lbrace \dfrac{\Vert \bff_{\bfe A}\Vert}{\Vert \bff_{\bfe B}\Vert} : A\subset B, \vert B\vert\leq m, \bfe\in\Psi_B\right\rbrace.$$
\end{definition}

\begin{remark}\label{inequalities}\rm 
Strightforward from the above definitions, for all $m\in \N$, we have the following inequelities. The nontrivial calculation in (ii) can be found in \cite[Lemma~2.1]{BBG}.
\begin{enumerate}[\upshape (i)] 
\item  $\vert g_m-g_m^c\vert\leq 1$, 
\item $\tilde{g}_m\leq \min\lbrace 2\min\lbrace g_m, g_m^c\rbrace,\, g_m g_m^c\rbrace$,
\item $\gamma_m\le \min\{g_m, g_m^c\}$.
\end{enumerate}
\end{remark}
\smallskip

\begin{definition}[\cite{AABW}]
An M-basis $\mathcal B$ is $C_{ql}$-\textbf{quasi-greedy for largest coefficients} if $C_{ql}:=\sup_m \mathbf q_m<\infty$, where
$$
\mathbf q_m := \sup\left\lbrace \dfrac{\Vert \bff_{\bfe A}\Vert}{\Vert x+\bff_{\bfe A}\Vert} :\  \vert A\vert \leq m,\ A\cap\supp(x)=\emptyset,\  \max_n \vert \be_n^*(x)\vert \leq 1,\ \bfe\in\Psi_A\right\rbrace.
$$
\end{definition}

\begin{definition}\label{definitionsuperconservative}
An M-basis $\mathcal{B}$  is $C_{sc}$-\textbf{superconservative} if $C_{sc}:=\sup_m \mathbf{sc}_m<\infty$, where $\mathbf{sc}_m$ is the $m$th \textbf{superconservative parameter} of a basis defined as follows:
$$\mathbf{sc}_m := \sup\left\lbrace \dfrac{\Vert\bff_{\bfe A}\Vert}{\Vert \bff_{\bfe' B}\Vert}:\ \vert A\vert \le \vert B\vert\leq m,\ A\leq m, A<B,\ \bfe\in\Psi_A,\ \bfe'\in\Psi_B\right\rbrace.
$$

\noindent When $\bfe\equiv\bfe'\equiv 1$, we write $\mathbf{c}_m$ instead of $\mathbf{sc}_m$ and call this constant the  $m$th   \textbf{conservative parameter}. Also, we say that $\mathcal{B}$ is $C_{c}$-\textbf{conservative} if $C_{c}:=\sup_m \mathbf{c}_m<\infty$.
\end{definition}

\begin{remark} \label{remarknaturalextension} {\rm Conservative (Schauder) bases were introduced in \cite{DKKT}: a basis is conservative with constant $C$ if $\| \bff_{ A}\|\le C \| \bff_{ B}\|$ whenever $A$ and $B$ are finite sets with $A<B$ and $|A|\le |B|$. For such bases, our definition of a conservative basis is equivalent, with the same constant. Indeed, if $\mathcal{B}$ is $C_c$-conservative, by Definition~\ref{definitionsuperconservative},  given $A<B$ finite sets with $|A|\le |B|$,  for any $m\ge B$ we have 
$$
\| \bff_{ A}\|\le \mathbf{c}_m\| \bff_{ B}\|\le\sup_m \mathbf{c}_m \| \bff_{ B}\|= C_c\| \bff_{ B}\|, 
$$
so $\mathcal{B}$ is also $C_c$-conservative as defined in \cite{DKKT}. On the other hand, if $\mathcal{B}$ is $C$-conservative according to the definition in~\cite{DKKT}, it is immediate that $\mathbf{c}_m\le C$ for all $m$, so $C_{c}\le C$.}
\end{remark}

\begin{remark} \rm In \cite{DKO2015}, the authors define a more restrictive  $m$th conservative parameters $\mathbf{c}(m)$. The difference between our definition and that of \cite{DKO2015} is that in the latter, the supremum is taken over all finte sets $A$ and $B$ such that $|A|=|B| \le m$ and $A\leq m <B$  which entails the additional conditions of taking $m<B$ and $|A|=|B|$. For the purposes of this article, we find the definition of $\mathbf{c}_m$ more convenient, not only because its definition is less restrictive than that of $\mathbf{c}(m)$, but also because (by Remark \ref{remarknaturalextension}) the sequence ($\mathbf c_m$) allows us to recover the original definition of conservative bases of \cite{DKKT}.
\end{remark}

Finally, we introduce the following concept, which will  be used to study the behaviour of $\widehat{\mathbf L}_m$.

\begin{definition}\label{definitionpslc}
An M-basis $\mathcal B$ is $C_{pl}$-\textbf{partially symmetric for largest coefficients} ($C_{pl}$-PSLC for short) if $C_{pl}:=\sup_m \mathbf \omega_m<\infty$, where
\begin{eqnarray*}
\omega_m := \sup
\left\lbrace 
\dfrac{\Vert x+t\bff_{\bfe A}\Vert}{\Vert x+t\bff_{\bfe' B}\Vert}: \quad 
\begin{array}{ll}  
\vert A\vert \leq \vert B\vert\leq m, & A<\supp(x)\cupdot B, A\le m, \\ 
 \bfe\in\Psi_A,\bfe'\in\Psi_B, & |t| \geq\max_n\vert\be_n^*(x)\vert
 \end{array}
\right\rbrace.
\end{eqnarray*}
\end{definition}
\smallskip
Notice that  $\omega_m\ge 1$ for all $m$, which is clear from  the  definition if we take $A=B=\emptyset$.

\begin{remark}\label{remarkoriginalpslc} 
{\rm Note that a slight variant of the argument of Remark~\ref{remarknaturalextension} shows that a basis is $C_{pl}$-PSLC if and only if there is $K>0$ such that for any finite sets $A$ and $B$, any  $x \in \mathbb{X}$ with  $|A|\le |B|, A<\supp(x)\cupdot B$,  for $|t|\ge \max_{n}\vert \be_n^{*}(x)\vert$, $\bfe\in\Psi_A,\bfe'\in\Psi_B$, 
$$
\Vert x+t\bff_{\bfe A}\Vert \le K \Vert x+t\bff_{\bfe' B}\Vert, 
$$ 
and $C_{pl}$ is the minimum $K$ for which this inequality holds. }
\end{remark}

The parameter $\omega_m$ is a weaker version of the parameter $\nu_m$ used to define the constant associated to Property (A) that appears in \cite{AW, BBG, DKOSZ}: a  basis has the $C_a$-Property (A) (or is $C_a$-symmetric for largest coefficients) if $C_a:=\sup_m \nu_m<\infty$, where
\begin{eqnarray}\label{sym}
\nu_m:= 
\sup\hskip -.1cm\left\lbrace 
\dfrac{\Vert x+t\bff_{\bfe A}\Vert}{\Vert x+t\bff_{\bfe' B}\Vert}: 
\begin{array}{l}  
\vert A\vert \leq \vert B\vert\leq m,\  A\cap B=\emptyset,\ \supp(x)\cap (A\cup B)=\emptyset \\ 
 \vert t\vert\geq\max_n\vert\be_n^*(x)\vert,\  \bfe\in\Psi_A,\ \bfe'\in\Psi_B
\end{array}
\hskip -.15cm \right\rbrace\hskip -.1cm.
\end{eqnarray}
\vskip .2cm

Among the definitions given above, the parameters $g_m, g_m^{c}, \tilde{g}_m, \mathbf{q}_m, \gamma_m,  \mathbf{L}_m, \widetilde{\mathbf{L}}_m$ and $\mathbf{L}_m^{re}$ are well-known in the literature and were studied, for instance, in \cite{AABW,BBG,BBGHO2,DKKT,DKO2015,Wo}.

Our main results for the strong residual parameters $(\widehat{\mathbf L}_m)_m$ are the following.

\begin{proposition}\label{triv}
	For every $m=1,2,\dots$,
	$$\widehat{\mathbf L}_m \leq 1+2\kappa m,$$
	where $\kappa = \sup_{m,n} \Vert \be_m\Vert\Vert\be_n^*\Vert$.
\end{proposition}

\begin{theorem}\label{one}
	For each $m=1,2,\dots$,
	\begin{eqnarray*}
	\widehat{\mathbf L}_m \leq  g_m^c+\tilde g_m\mathbf{sc}_m.
	\end{eqnarray*}
\end{theorem}
\begin{proposition}\label{propositionnewdef}For each $m=1,2,\dots$,
$$
\omega_m\le \max_{1\le k\le m}\widehat{\mathbf L}_k.
$$
\end{proposition}

\begin{theorem}\label{one2y3}
	For each $m=1,2,\dots$, 
	\begin{eqnarray*}
 g_{m}^c\le \ \widehat{\mathbf L}_m \le g_{m-1}^c \omega_m.
 \end{eqnarray*}
Moreover, $\widehat{\mathbf L}_1= \omega_1.$
\end{theorem}

On the other hand, we characterize strong partially greedy bases in terms of partial greediness and quasi-greediness, and $1$-strong partially greedy bases in terms of the $1$-PSLC property. Also, we prove that $1$-partially greedy bases are strong partially greedy. 

\begin{theorem}\label{two}
An M-basis $\mathcal B$ is $1$-strong partially greedy  if and only if  $\mathcal B$  is $1$-partially symmetric for largest coefficients.
\end{theorem}
The above result is an improvement with respect to Theorem \ref{one2y3}, since it allows us to deduce that PSLC with $C_{pl}=1$ implies quasi-greediness with $C_q=1$. 

The paper is structured as follows: in Section \ref{pre}, we give some basic results that will be used later. In Section \ref{main}, we prove Theorem \ref{one} and the respective corollaries, whereas in Section \ref{onep}, we prove Theorem \ref{two} and give a characterization of $1$-PSLC bases. In Section~\ref{sectionextensions}, we discuss the relation between the concepts of partial greediness and strong partial greediness, and extend \cite[Theorem 3.2]{DKKT} to the context of Markushevich bases.

\section{Preliminary results}\label{pre}

This section is devoted to giving some different estimates of the  parameters $\omega_m$ and $\mathbf{sc}_m$. In order to do so we
recall the definition of the \textit{truncation operator}, which was first considered in \cite{DKK}. We use some of its properties that connect this operator with the quasi-greedy parameter. Also, we appeal to several lemmas given in \cite{BBG} that we list below for the sake of the reader.

\subsection*{Truncation Operator}
For each $t>0$, we define the $t$-truncation of $x\in\mathbb F$ by

\[
T_t(x)=\begin{cases}
t\,\sgn(x)\quad   & \textit{if}\; \vert x\vert > t,\\
x& \textit{if}\; \vert x\vert \le t.
\end{cases}
\]
 \
We can extend $T_t$ to an operator in $\X$  - which we still call $T_t$ -  by formally assigning $T_t(x)\sim \sum_{n=1}^\infty T_t(\be_n^*(x))\be_n$, that is,
$$
T_t(x):=t\bff_{\bfe \Lambda_{t}(x)}+(I-P_{\Lambda_{t}(x)})(x),
$$
where $\bfe = \lbrace \sgn(\be_n^*(x))\rbrace$ and $\Lambda_{t} (x):= \{ n\in\N: \vert \be_n^*(x)\vert>t\}$ is the $t$-index set for $x$ associated to $T_t$. Since $\Lambda_{t}(x)$ is finite, $T_t \colon \X \to \X$ is well defined.

\begin{lemma}\cite[Lemma 2.3]{BBG}\label{bb1}
Let $x\in\X$ and $\bfe = \lbrace \sgn(\be_n^*(x))\rbrace$. For each $m$-greedy set $A$ of $x$,
	$$\min_{n\in A}\vert\be_n^*(x)\vert \Vert \bff_{\bfe A}\Vert \leq \tilde g_m \|x\|.$$
\end{lemma}

With the notation above, we also have:

\begin{lemma}\cite[Lemma 2.5]{BBG}\label{bb2}
	For all $t>0$ and $x\in\X$, 
$$
	\Vert T_t(x)\Vert \leq g_{|\Lambda_{t} (x)|}^c\Vert x\Vert.
$$
Moreover, if  $\Lambda_{t} (x)=\emptyset$, the equality is attained as    
$T_t(x)=x$, $|\Lambda_{t} (x)|=0$ and $g_0^c=1$.
\end{lemma}

Given a subset $E\subset \X$, its convex hull is denoted by $\co(E)$. 
	\begin{lemma}\label{con}\cite[Lemma 2.7]{BBG}
		For every finite set $A\subset\N$, we have
		$$\co(\lbrace \bff_{\bfe A} : \bfe\in\Psi_A\rbrace) = \Big\{\sum_{n\in A}z_n \be_n : \vert z_n\vert\leq 1\Big\}.$$
\end{lemma}
\medskip

Now we focus on the task of giving estimates of the parameters $\mathbf{sc}_m$  and $\omega_m$ in terms of known parameters appearing in \cite{AW, AABW, BBG, DKOSZ}. Our estimates allow us to characterize bases which are PSLC (see Corollary~\ref{PSLC via estimates}).

\begin{proposition}\label{mejora}
	For each $m=1,2,\dots$, 
	$$\mathbf{c}_m\leq \mathbf{sc}_m \leq 4\nu^2\gamma_m\mathbf{c}_m,$$
	where $\nu=1$ or $2$ if $\mathbb F=\mathbb R$ or $\mathbb C$, respectively.
\end{proposition}

\begin{proof}
The lower bound for $ \mathbf{sc}_m$ is trivial by definition. To prove the upper bound, consider first that $\mathbb{F}=\mathbb{R}$ and take $A<B$\txcr{,} $A\le m$, $\vert A\vert\le \vert B\vert\le m$ and $\bfe\in\Psi_A$, $\bfe'\in\Psi_B$. Assume first that $\bfe \in\lbrace \pm 1\rbrace^{|A|}$. Then, if  $A^+=\lbrace n\in A : \e_n = 1\rbrace$ and $A^- = A\setminus A^+$, 
\begin{equation}\label{onda2}
\Vert \bff_{\bfe A}\Vert \le \Vert \bff_{A^+}\Vert + \Vert \bff_{A^-}\Vert\leq \mathbf{c}_m\Vert \bff_{B}\Vert + \mathbf{c}_m\Vert \bff_{B}\Vert=2\mathbf{c}_m\Vert \bff_{B}\Vert\le  4 \mathbf{c}_m\gamma_m \Vert \bff_{\bfe' B}\Vert,
\end{equation}
where the last inequality follows by \cite[Lemma 3.3]{BBG}.  For $\bfe \not \in\lbrace \pm 1\rbrace^{|A|}$, by convexity we have
$$
\max{\{\|\sum\limits_{n\in A}\Re{(\e_n)}\be_n\|, \|\sum\limits_{n\in A}\Im{(\e_n)}\be_n\|\}}\le \max_{\eta \in \lbrace \pm 1\rbrace^{|A|}} \|\bff_{\eta A}\|.
$$
Thus, by the triangle inequality we get $\|\bff_{\bfe A}\|\le 2 \|\bff_{\eta A}\|$ for some $\eta \in\lbrace \pm 1\rbrace^{|A|}$, and the result follows by \eqref{onda2}.
\end{proof}

\begin{Remark}
\rm 
Proposition~\ref{mejora} shows that the superconservative parameter can be controlled by the parameter $\gamma_m$ of Definition~\ref{C_u:unconditionall for cte coeff} and the conservative parameter. This implies that a basis that is conservative and unconditional for constant coefficients is superconservative. We do not know whether every superconservative basis is unconditional for constant coefficients, but as Example~\ref{examplenotcontrolled} shows, the parameter $\gamma_m$ cannot be controlled by $\mathbf{sc}_m$. 
\end{Remark}
\begin{example}\label{examplenotcontrolled}
	Let $\mathbb \X$ be the completion of $c_{00}$ under the norm
	$$
	\Vert (a_n)_n\Vert:= \sup_{D\in  \mathcal S}|\sum_{n\in D} a_n|,
	$$
	where 
	$$
\mathcal S:=\lbrace D\subset\mathbb N\colon  0<\vert D\vert \leq \sqrt{\min D}\text{, and } i<j<k, i\in D, k\in D \Longrightarrow j\in D \rbrace.
	$$
	There is no constant $K$ such that
	$$
	\gamma_m\le K \mathbf{sc}_m,\;\forall m\in \mathbb{N}.
	$$
\end{example}
\begin{proof}
	It is clear that the canonical basis $(\be_n)_n$ is a normalized Schauder basis for $\mathbb{X}$.\\
	Note that if $D\in \mathcal S$, then $D$ is a nonempty interval in $\mathbb{N}$, that is, either $D=\{m\}$ for some $m\in \mathbb{N}$, or $D=\{m,m+1, \dots,m+n\}$ for some $m,n\in \mathbb{N}$. Moreover, any nonempty interval $D'\subset D$ is also an element of $\mathcal S$. This guarantees that for every $x$ with finite support, 
	$$
	\|x\|=\max_{\substack{D \in \mathcal{S}\\\min \supp{(x)}\le D\le \max{\supp{(x)}} }}\Big\{|\sum\limits_{n=\min D}^{\max D}\be_n^{*}(x)|\Big\},
	$$
where $(\be_n^{*})_n$ are the biorthogonal functionals corresponding to $(\be_n)_n$.
In particular, the basis is monotone. Note also that, since $\lbrace n\rbrace\in\mathcal S$ for all $n\in\N$, for all $x\in\X$ we have
	$$
	\Vert x\Vert \geq \sup_{n\in \N}\vert\be_n^*(x)\vert=\Vert x\Vert_{\infty}.
	$$
	\\
	Now fix $m\in \mathbb N$,  a finite set \txcr{$A$} with $A\le m$, and $\bfe \in \Psi_A$. Then, for any $D\in \mathcal S$ such that $D\le \max{A}$, we have 
	$$
	|D|\le \sqrt{\min D}\le \sqrt{\max A}\le \sqrt{m}.
	$$
	Thus,
	$$
	\|\mathbf{1}_{\varepsilon A}\|\le \sup_{\substack{D\in \mathcal S\\D\le m}}\Big\{\sum_{n\in D} |\be_n^*(\mathbf{1}_{\bfe A})|\Big\}\le \sqrt{m}. 
	$$
	Given that for any finite set $B$ and any $\bfe' \in \Psi_B$, 
	$$
	1=\|\mathbf{1}_{\bfe' B}\|_{\infty}\le \|\mathbf{1}_{\bfe' B}\|,
	$$
	it follows that 
	$$
	\mathbf{sc}_m\le \sqrt{m}, \;\forall m\in \mathbb N. 
	$$
	Now fix $m\in \mathbb{N}$, and let 
	$$
	B\colon=\{4m^2+1,\dots, 4m^{2}+2m\},\qquad A\colon=\{n\in B: n\text{ is even}\}. 
	$$
	Clearly $|B|=2m$, $B\in \mathcal S$, and $|A|=m$. Now define $\bfe \in \Psi_B$ as follows 
	$$
\e_n\colon=(-1)^{n}\;\forall n\in B. 
	$$
Since $B$ is an interval, for any $D\in \mathcal S$ with $4m^{2}+1\le D\le 4m^2+2m$ we have
	$$
	|\sum\limits_{n=\min {D}}^{\max {D}}\be_n^*(\mathbf{1}_{\bfe B})|=|\sum\limits_{n=\min{D}}^{\max{D}}(-1)^{n}|\le 1=\|\mathbf{1}_{\bfe B}\|_{\infty}.
	$$
	Hence, 
	$$
	\|\mathbf{1}_{\bfe B}\|=1. 
	$$
On the other hand, since $B\in \mathcal S$ and $A\subset B$, 
	\begin{eqnarray*}
		\|\mathbf{1}_{\bfe A}\|& =\sup_{D\in \mathcal S} \big\{|\sum_{n\in D} \be_n^*(\mathbf{1}_{\bfe A})|\big\} \ge |\sum\limits_{n\in B}\be_n^*(\mathbf{1}_{\bfe A})|=|\sum\limits_{n\in A}\be_n^*(\mathbf{1}_{\bfe A})|=|\sum\limits_{n\in A}1|=m. 
	\end{eqnarray*}
	Thus,
	$$
	\gamma_{2m}\ge \frac{\|\mathbf{1}_{\varepsilon A}\|}{\|\mathbf{1}_{\varepsilon B}\|}=m,
	$$
	and the assertion is proved.
\end{proof}

\begin{proposition}\label{one4}
	For each $m=1,2,\dots$,
	\begin{eqnarray}
	\max\left\lbrace \frac{\mathbf q_m}{2},\, \mathbf{sc}_m\right\rbrace \leq \omega_m\leq 1+\mathbf q_m(1+\mathbf{sc}_m).
	\end{eqnarray}
\end{proposition}

\begin{proof}
First we prove the lower bound.  It is clear that $\mathbf{sc}_m\leq \omega_m$. To prove that $\mathbf q_m \leq 2\omega_m$, take $A, x$ and  $\bfe$ as in the definition of $\mathbf q_m$. Then,
\begin{eqnarray*}\label{prob1}
\Vert \bff_{\bfe A}\Vert&\leq& \Vert x+\bff_{\bfe A}\Vert+\Vert x\Vert\leq \Vert x+\bff_{\bfe A}\Vert+\omega_m\Vert x+\bff_{\bfe A}\Vert\\
&\leq& 2\omega_m\Vert x+\bff_{\bfe A}\Vert.
\end{eqnarray*}
To prove the upper bound, let $A, B, x, t, \bfe$ and $\bfe'$ be as in the definition of $\omega_m$, that is, $A<\supp(x)\cupdot B$, $\vert A\vert\leq \vert B\vert\leq m$, $A\le m$, $t\geq\max_n\vert\be_n^*(x)\vert$ and $\bfe\in\Psi_A$, $\bfe'\in\Psi_B$. We have
	
\begin{eqnarray*}\label{prob2}
\Vert x+t\bff_{\bfe A}\Vert &\leq& \Vert x+t\bff_{\bfe' B}\Vert+\Vert t\bff_{\bfe A}\Vert+\Vert t\bff_{\bfe' B}\Vert\nonumber\\
&\leq& \Vert x+t\bff_{\bfe' B}\Vert+\mathbf{sc}_m\Vert t\bff_{\bfe' B}\Vert+\mathbf q_m\Vert x+t\bff_{\bfe' B}\Vert\\
&\leq& (1+\mathbf q_m + \mathbf q_m\mathbf{sc}_m)\Vert x+t\bff_{\bfe' B}\Vert,\nonumber
\end{eqnarray*}
which concludes the proof. 
\end{proof}

\begin{corollary}\label{PSLC via estimates}
An M-basis $\mathcal B$ in a Banach space $\X$ is partially symmetric for largest coefficients if and only if $\mathcal B$ is superconservative and quasi-greedy for largest coefficients.
\end{corollary}

\begin{proof}
The lower bound of Proposition~\ref{one4} shows that a basis that is $C_{pl}$-PSLC is also superconservative with constant $C_{sc}\le C_{pl}$ and quasi-greedy for largest coefficients with constant $C_{ql} \le 2  C_{pl}$.
On the other hand, if the basis is $C_{sc}$-superconservative and $C_{ql}$-quasi-greedy for largest coefficients, applying the upper bound of Proposition~\ref{one4} we obtain 
$$C_{pl}=\sup_m \omega_m \leq 1+C_{ql}(1+C_{sc}).$$
Hence, the basis is PSLC.
\end{proof}

\section{Main results for $\widehat{\mathbf L}_m$}\label{main}
In this section, we prove the main results concerning $\widehat{\mathbf L}_m$. 
\begin{proof}[Proof of Proposition \ref{triv}:]
The proof of this result is immediate using 
that by definition
$$
\widehat{\mathbf L}_m \leq \widetilde{\mathbf L}_m, 
$$
and also that $\widetilde{\mathbf L}_m \leq 1+2\kappa m$, by~\cite[Theorem 1.8]{BBG}.
\end{proof}

\begin{proof}[Proof of Theorem \ref{one}:]
Take $x\in\mathbb X$, fix $m$ and take a greedy operator of order $m$, $G_m$.  Set $A:=\supp (G_m(x))$, take $k\leq m$ and define $D:=\{1,\dots,k\}$. We have the following 
possibly trivial decomposition:
	
\begin{eqnarray*}
	x-G_m(x)=P_{(A\cup D)^c}(x)+P_{D\setminus A}(x).
\end{eqnarray*}
	
\noindent On the one hand, since 
$$
P_{(A\cup D)^c}(x) = (I-P_{A\setminus D})(x-P_k(x)),
$$
we have 
\begin{eqnarray}\label{rr3}
\Vert P_{(A\cup D)^c}(x)\Vert \leq g_{m}^c \Vert x-P_k(x)\Vert.
\end{eqnarray}
	
\noindent On the other hand, notice that $D\setminus A \le k < A\setminus D$ and  $ A\setminus D$ is a greedy set for $x-P_k(x)$. Thus,  applying first Lemma \ref{con} and then Lemma \ref{bb1} with $\bfe=(\sgn(\be_n^*(x)))_n$, we obtain
\begin{eqnarray}\label{rr4}
\nonumber\Vert P_{D\setminus A}(x)\Vert &\leq& \mathbf{sc}_m \max_{n\in D\setminus A}\vert \be_n^*(x)\vert\Vert \bff_{\bfe (A\setminus D)}\Vert\\
\nonumber &\leq& \mathbf{sc}_m \min_{n\in A\setminus D}\vert \be_n^*(x-P_k(x))\vert\Vert \bff_{\bfe (A\setminus D)}\Vert\\
&\leq& \tilde g_{|A\setminus D|}\mathbf{sc}_m \Vert x-P_k(x)\Vert\\
\nonumber &\leq& \tilde g_m\mathbf{sc}_m \Vert x-P_k(x)\Vert.
\end{eqnarray}
As \eqref{rr3} and \eqref{rr4} hold for any $k\le m$,  a direct combination of both inequalities shows that $\|x-G_m(x)\|\le ( g_{m}^c+ \tilde g_m\mathbf{sc}_m)\widehat{\sigma}_m(x)$ and therefore, the result follows.
\end{proof}
\smallskip

\begin{corollary}\label{cor1}
If an M-basis is $C_q$-quasi-greedy, then for each $m=1,2,\dots$,
$$\widehat{\mathbf L}_m \leq C_q + 2C_q\mathbf{sc}_m.$$
\end{corollary}

\begin{proof}
	Just apply Theorem~\ref{one} and use the estimate $\tilde g_m\le 2C_q$ for all $m\in\N$. 
\end{proof}
\smallskip

\begin{proof}[Proof of Proposition~\ref{propositionnewdef}:]
Take $A, B, x, \bfe, \bfe'$ as in the definition of $\omega_m$.  A careful look at this definition allows us to only consider $t>0$. Now, let $m_1:=\max A$. Since $|A|\le |B|\le m$, $A\le m_1\le m$ and $A<B$, there  
exists a possibly empty set $D$ such that $D\subset \{1,\dots, m_1\}\setminus A$ 
and $m_1\le \vert D\cup B\vert\le m$. Let $m_2:=\vert D\cup B\vert$ and define, for $\eta>0$, the element 
$$
y:= x+t\bff_{\bfe A}+(t+\eta)(\bff_{\bfe' B}+\bff_{ D}).
$$

\noindent As in particular $A\cap \supp(x)$ is the empty set and $t\ge\max_n |\be^*_n(x)|$, we see that $G_{m_2}(y)= (t+\eta)(\bff_{\bfe' B}+\bff_{ D})$ and $P_{m_1}(y)=t\bff_{\bfe A}+(t+\eta)\bff_{ D}$, then

$$
\Vert x+t\bff_{\bfe A}\Vert = \Vert y-G_{m_2}(y)\Vert \leq \widehat{\mathbf L}_{m_2}\Vert y-P_{m_1}(y)\Vert \le 
\max_{1\le k\le m}\widehat{\mathbf L}_k\Vert x+(t+\eta)\bff_{\bfe' B}\Vert.
$$

\noindent Since the inequality holds for any $\eta>0$, we conclude that $
\omega_m\le \max_{1\le k\le m}\widehat{\mathbf L}_k.$
\end{proof}
\medskip

The following lemma will be used to prove Theorem \ref{one2y3}. It is   based on \cite[Lemma~2.8]{BBG}, but stated for the parameter $\omega_m$ instead of $\nu_m$. We give a proof for the sake of completeness. \\
\begin{lemma}\label{pl}
Let $x\in\X$ and $|t|\geq \max_{n}\vert\be_n^*(x)\vert$. Then,
$$
\Vert x+z\Vert \leq \omega_m \Vert x+t\bff_{\bfe B}\Vert,
$$
for any finite set $B$, $\vert B\vert\le m$,  any $\bfe\in\Psi_B$, and any $z$ such that $\supp(z)\leq m$, $\supp{(z)}< B\cupdot\supp(x)$ and $|t|\ge\max_n\vert\be_n^*(z) \vert$.
\end{lemma}
\begin{proof}
Notice that by the definition of $t\bff_{\bfe B}$ with  $\bfe \in\Psi_B$, it is enough to give a proof for $t>0$. By the definition of the parameter $\omega_m$, the result is true if we take $z=\bff_{\tilde \bfe A}$, for any $\tilde\bfe\in\Psi_A$, $A< \supp(x)\cupdot B $, $A\le m$, and $\vert A\vert\leq\vert B\vert\leq m$. Thanks to the convexity of the norm, it continues to be true for any element $z\in \co\left(\lbrace \bff_{\tilde \bfe A}: \tilde \bfe\in\Psi_A\rbrace\right)$. Then, the general case follows from Lemma \ref{con}.
\end{proof}

\begin{proof}[Proof of Theorem~\ref{one2y3}:]
Taking $k=0$ in the definition of $\widehat{\sigma}_m(x)$, we see that $\widehat \sigma_m(x)\le \|x\|$ for all $x$. Then, the inequality $g_m^c \leq \widehat{\mathbf L}_m$ is immediate for all $m\in \N$.\\
To prove the upper bound, fix $m\in \N$,  $x\in\mathbb X$ and a greedy operator of order $m$, $G_m$. We will show that  $\| x - G_m(x)\| \le g_{m-1}^c \omega_m\| x - P_k(x)\|$ for all $k\le m$.\\
Set $A:=\supp G_m(x)$, fix  $k\leq m$ and take $D:=\{1,\dots,k\}$.  If $A=D$, then $k=m$ and $G_m(x)=P_m(x)$; as $g_{m-1}^c, \omega_m\ge 1$ there is nothing to prove. If $A\not=D$, consider the following decomposition:

\begin{eqnarray}\label{rr5a}
x-G_m(x)=P_{(A\cup D)^c}(x-P_k(x))+P_{D\setminus A}(x).
\end{eqnarray}

\noindent Applying  Lemma \ref{pl} with  $t=\min_{n\in A}\vert \be_n^*(x)\vert$\txcr{,}  $z= P_{D\setminus A}(x)$, and $\bfe\in \Psi_{A\setminus D}$ such that $\varepsilon_n=\sgn(\be_n^*(x))$, we have
	\begin{eqnarray}\label{plone5a}
	\Vert P_{(A\cup D)^c}(x-P_k(x))+P_{D\setminus A}(x)\Vert\leq \omega_m \Vert P_{(A\cup D)^c}(x-P_k(x))+t\bff_{\bfe A\setminus D}\Vert. 
	\end{eqnarray}

\noindent Let $T_t$ be the $t$-truncation operator and $\Lambda_{t}$ be its associated index set. Notice that $|\Lambda_{t}(x-P_k(x))|\le m-1$ and 
$$
P_{(A\cup D)^c}(x-P_k(x))+t\bff_{\bfe A\setminus D} = T_t(x-P_k(x)).
$$
Thus, an application of Lemma \ref{bb2}  yields
	\begin{eqnarray}\label{rtwo5a}
	\Vert P_{(A\cup D)^c}(x-P_k(x))+t\bff_{\bfe A\setminus D}\Vert \le g_{|\Lambda_{t}(x-P_k(x))|}^{c} \Vert x-P_k(x)\Vert\le  g_{m-1}^c\Vert x-P_k(x)\Vert.
\end{eqnarray}

\noindent Hence, the upper bound in the statement follows combining  \eqref{rr5a},  \eqref{plone5a} and \eqref{rtwo5a}. 
Finally, by Proposition~\ref{propositionnewdef} we have $\omega_1\le \widehat{\mathbf{L}}_1$, which completes the proof.
\end{proof}

{\begin{Remark}\rm 
Note that the sequence $(g_m^{c})_m$ is increasing, so from Theorem~\ref{one2y3} we also get, \txcr{for all $m\in \N$,}
$$
g_m^{c}\le \widehat{\mathbf L}_m\le g_m^{c}\,\omega_m.
$$
\end{Remark}

We end this section presenting two examples that allow us to study the optimality of the inequalities of the main results about $\widehat{\mathbf L}_m$. First note that if $\mathcal B$ is the unit vector basis of $\mathtt{c}_0$ or $\ell_p$ with $1\le p<\infty$, we have $g_m^{c}=1$ and $\omega_m=1$ for all $m\in \N$. Hence, equality holds throughout in Theorem~\ref{one2y3}. With the next example  we show that equality holds in Propositions~\ref{propositionnewdef} and~\ref{triv}.

\begin{example}\label{ex:igualdades en ctes}
Let $\mathbb{X}$ be the completion of $\mathtt{c}_{00}$ under the norm
$$
\Vert \mathbf (a_n)_n\Vert:= \sup_{n\geq 1}\Big\vert \sum_{j=1}^n a_j\Big\vert.
$$

The canonical  basis of $\mathbb{X}$, $(\be_n)_n$, is monotone and the following hold for all $m\in \mathbb N$:
\begin{enumerate}[\upshape (i)]
\item $g_m= \tilde g_m= 2m$ and $g_m^c = 2m+1$.
\item $\mathbf{sc}_m = m$ and $\mathbf{c}_m=1$.
\item $\omega_m = \widehat{\mathbf L}_m =\max_{1\le k\le m}\widehat{\mathbf{L}}_k= 1+4m$.
\end{enumerate}
Then, equality holds in Propositions~\ref{triv} and~\ref{propositionnewdef}.
\end{example}

\begin{proof}
A proof of (i) can be found in \cite[Proposition 5.1]{BBG}. To see (ii), notice that for every finite set $A$,
$$
\Vert \bff_A\Vert = \vert A\vert.
$$
Hence, $\mathbf{c}_m = 1$. Also, notice that for $\bfe\in\Psi_A$ we have
$$
1\leq \Vert \bff_{\bfe A}\Vert \leq \vert A\vert.
$$
Then, for $A, B, \bfe$, and $\bfe'$ as in the definition of  $\mathbf{sc}_m$, we have 
$$
\frac{\Vert \bff_{\bfe A}\Vert}{\Vert \bff_{\bfe' B}\Vert} \leq \vert A\vert\le m.
$$
Finally with $A:=\{1,\ldots,m\}$, $B:=\{m+1,\ldots, 2m\}$,  $\bfe\equiv 1$ and $\bfe'=((-1)^n)_n$, and noting that $\Vert \sum_{n=1}^m (-1)^n \be_n\Vert = 1$, we see that the bound above is attained and $\mathbf{sc}_m = m$.
\medskip

\noindent Next, let us show that (iii) holds. Taking into account that $\Vert \be_n\Vert=1$ for all $n\in\N$, $\Vert \be_1^*\Vert=1$ and $\Vert \be_n^*\Vert=2$ for all $n\geq 2$, the constant  in Proposition~\ref{triv} is $\kappa=2$, so we obtain that $\widehat{\mathbf L}_k \leq 1+4k$ for any $k$. Then, clearly, $\d\max_{1\le k\le m}\widehat{\mathbf{L}}_k\le 1+4m$. \\
To show the lower bound, consider the sets $A$ and $B$ with $|A|=|B|=m$ so that 
$$
\begin{array}{rl}
\bone_A:= &  \big(\underbrace{1,\dots,1}_{m},0,0,\ldots),\\
\bone_B:= & \big(\underbrace{0,\dots,0}_{m},\underbrace{0,1,0}\;,\;\ldots\;,\;\underbrace{0,1,0}\;, \;0,\ldots\big),
\end{array}
$$
and define $x\in \X$ with $|\supp(x)|=2m+1$  as follows
$$
\begin{array}{rl}
x:= &  \big(\underbrace{0,\dots,0}_{m},\underbrace{\overbrace{\tfrac12,0,\tfrac12}\;,\;\ldots\;,\;\overbrace{\tfrac12,0,\tfrac12}\;, \;\tfrac12}_{\supp(x)},0,0,\ldots\big).
\end{array}
$$
Then $  \Vert x+\bff_A\Vert= 2m +\frac 12$ and  $\Vert x-\bff_{B}\Vert=\frac 12$. Now, using Proposition~\ref{propositionnewdef} we get 
$$
\max_{1\le k\le m}\widehat{\mathbf{L}}_k\geq \omega_m \geq \dfrac{\Vert x+\bff_A\Vert}{\Vert x-\bff_{B}\Vert}=1+4m,
$$
and the proof is complete.
\end{proof}

Next, we prove the optimality of the estimate of Theorem~\ref{one} and the right-hand side of the inequality in Theorem~\ref{one2y3}. Additionally, we give another example in which equality in Proposition~\ref{propositionnewdef} holds. 

\begin{example}\label{exell1c0} Let $(\be_n)_n$ and $(\mathbf{f}_n)_{n}$ be the unit vector bases of $\ell_1$ and $\mathtt{c}_0$ respectively, and let $\mathbb{X}$ be the space $\ell_1\times \mathtt{c}_0$ with the norm 
$$
\|(x,y)\|:=\max\{\|x\|_{1},\|y\|_{\infty}\}.
$$
For each $m\in \N$, define 
$$
\mathbf{x}_{2m-1}:=(\be_m,0),\qquad \mathbf{x}_{2m}:=(0,\mathbf{f}_m).
$$
Then, $(\mathbf{x}_n)_n$ is a Schauder basis for $\mathbb{X}$, and the following hold for all $m\in \N$:
\begin{enumerate}[\upshape (i)]
\item \label{quasigreedy} $g_m= \tilde g_m= g_m^c = 1$.
\item \label{conservativeness} $\mathbf{sc}_{2m}=\mathbf{c}_{2m}=\mathbf{sc}_{2m-1}=\mathbf{c}_{2m-1}=m$.
\item \label{omegaex}$\omega_{2m}=\omega_{2m-1}=\widehat{\mathbf{L}}_{2m-1}=\widehat{\mathbf{L}}_{2m}=m+1.$ 
\end{enumerate}
Thus, equalities in Theorem~\ref{one} and Proposition~\ref{propositionnewdef} hold, and also the right-hand side of the inequality in Theorem~\ref{one2y3}. 
\end{example}

\begin{proof}
It is easy to check that $(\mathbf{x}_n)_n$ is a $1$-unconditional  basis for $\X$, so \eqref{quasigreedy} holds trivially. To prove \eqref{conservativeness}, fix $m\in \N$ and  take any nonempty set $A$, with $A\le 2m$  and $\bfe \in \Psi_A$. Define $A_e:=\{j\in A: j\text{ is even}\}$ and $A_o:=A\setminus A_e$. We have
$$
\|\bff_{\bfe A}\|=\max\{\|\bff_{\bfe A_o}\|_{1}, \|\bff_{\bfe A_e}\|_{\infty}\}\le \max\{|A_o|,1\}\le m. 
$$
It follows that 
\begin{equation}
\mathbf{sc}_{2m}\le m.\label{upperscm}
\end{equation}
Now, define the sets
$$
A_m:=\{2j-1\colon 1\le j\le m\},\qquad B_m:=\{2m+ 2j\colon 1\le j\le m\}. 
$$
Note that for all $m$, 
\begin{equation*}
|A_m|=|B_m|=m,\quad A_m \le 2m -1 <2m+2\le B_m,\quad \|\mathbf{1}_{A_m}\|=m, \quad \|\mathbf{1}_{B_m}\|=1. \label{somebounds}
\end{equation*}
In particular, this immediately gives 
\begin{equation}
\mathbf{c}_{2m-1}\ge m. \label{lowercm}
\end{equation}
Since $\mathbf{c}_{n}\le \min{\{\mathbf{c}_{n+1},\mathbf{sc}_{n}\}}\le \mathbf{sc}_{n+1}$ for all $n\in \N$, combining \eqref{upperscm} and \eqref{lowercm} we obtain \eqref{conservativeness}.  Finally, to prove \eqref{omegaex} first notice that by \eqref{quasigreedy}, \eqref{conservativeness} and Theorem~\ref{one}, for all $m$ it follows that
\begin{equation}
\widehat{\mathbf{L}}_{2m-1}\le m+1,\qquad \widehat{\mathbf{L}}_{2m}\le m+1.\label{upperstrongpartially}
\end{equation}
Given that $(\omega_n)_n$ is an increasing sequence, \eqref{upperstrongpartially} and Proposition~\ref{propositionnewdef} yield  for all $m\in \N$,
\begin{equation}\label{upperomega}
\omega_{2m-1}\le \omega_{2m}\le \max_{1\le k\le 2m}\widehat{\mathbf{L}}_{k}\le m+1. 
\end{equation}
Considering $A_m$ and $B_m$ as before, with $\bfe'\in \Psi_{B_m}$ we have 
\begin{equation}
\|\bff_{A_m}+\mathbf{x}_{2m+1}\|=\|\sum\limits_{j=1}^{m+1}(\be_j,0)\|=m+1,\label{inell1}
\end{equation}
whereas
$$
\|\bff_{\bfe' B_m}+\mathbf{x}_{2m+1}\|=\max\{\|\be_{m+1}\|_{1}, \|\sum\limits_{j=1}^{m}\e'_{2m+2j} \mathbf{f}_{m+j}\|_{\infty}\}=1.
$$
Thus, by the choice of $A_m$ and $B_m$, it follows that for all $m\in \N$,  
\begin{equation}
\omega_{2m-1}\ge m+1. \label{omega+strongpartially}
\end{equation}
Now take $\eta>0$, and let 
$$
y_m:=\bff_{A_m}+\mathbf{x}_{2m+1}+(1+\eta)(\bff_{B_m}+\bff_{B_{2m}}).
$$
By \eqref{inell1}, we have 
$$
\|y_m-G_{2m}(y_m)\|=\|\bff_{A_m}+\mathbf{x}_{2m+1}\|=m+1, 
$$
and by the $1$-unconditionality of the basis, 
$$
\|y_m-G_{2m-1}(y_m)\|\ge \|y_m-G_{2m}(y_m)\|=m+1. 
$$
Since 
\begin{equation}
\|y_m-P_{2m-1}(y_m)\|=\|\mathbf{x}_{2m+1}+(1+\eta)(\bff_{B_m}+\bff_{B_{2m}})\|=(1+\eta) 
\end{equation}
and $\eta$ is arbitrary, it follows that
$$
\widehat{\mathbf{L}}_{2m}\ge m+1,\qquad \widehat{\mathbf{L}}_{2m-1}\ge m+1. 
$$
Combining these inequalities with \eqref{upperstrongpartially}, \eqref{upperomega} and  \eqref{omega+strongpartially}, we obtain \eqref{omegaex}. 
\end{proof}

\section{Characterization of $1$-strong partially greedy bases}\label{onep}

Since 2006, some authors have studied bases that have their respective greedy constants  attaining the least possible value, that is bases that have constant one.  For instance, F. Albiac and P. Wojtaszczyk \cite{AW} gave a characterization of $1$-greedy bases and F. Albiac and J. L. Ansorena \cite{AA1} characterized $1$-almost-greediness in terms of $1$-symmetry for largest coefficients  (see the paragraph below Remark~\ref{remarkoriginalpslc} for its definition). Following this spirit, in this section we prove a similar result for $1$-strong partially greedy and $1$-PSLC bases showing that this last condition is stronger than quasi-greediness. 

\begin{theorem}\label{equiv}
An M-basis $\mathcal B$ is $1$-strong partially greedy if and only if $\mathcal B$  is $1$-PSLC.
\end{theorem}
\begin{proof}
If $\widehat{\mathbf{L}}_m=1$ for all $m\in\N$, by Proposition~\ref{propositionnewdef} it follows that $\omega_m=1$ for all $m\in\N$. Reciprocally, if $\omega_m=1$ for all $m\in\N$, from the fact that $g_m^{c}\ge 1$ and Theorem~\ref{one2y3} we obtain
$$
\widehat{\mathbf{L}}_1= 1,\qquad 1\le g_m^{c}\le \widehat{\mathbf{L}}_m\le g_{m-1}^{c}\;\quad \text{for all\ } m\ge 2.   
$$
Thus, inductively it follows that $g_m^{c}=\widehat{\mathbf{L}}_m=1$ for all $m\in\N$. 
\end{proof}

We also have the following characterization of $1$-PSLC bases. 
\begin{proposition}\label{lemmashorter1PSLC} Let $\mathcal B=(\be_n)_{n}$ be an M-basis for a Banach space $\X$. The following are equivalent: 
\begin{enumerate}[\upshape (i)]
\item For any finite sets $A$ and $B$ such that $\vert A\vert\leq\vert B\vert$, $x\in\X$,  $A< \supp(x)\cupdot B$, $\bfe\in\Psi_A,\bfe'\in\Psi_B$, $t\in \mathbb{F}$ with $|t| \geq\max_{n}\vert\be_n^*(x)\vert$, 
\begin{equation}
\Vert x+t\bff_{\bfe A}\Vert\le \Vert x+t\bff_{\bfe' B}\Vert.\label{omega_m=1}
\end{equation}
\item $\mathcal B$ has the $1$-PSLC property.
\item \label{shorter} The following conditions hold: 
\begin{enumerate}[\upshape (I)]
\item \label{shortercondition} For all $x\in \X$, $t\in \mathbb{F}$ such that $|t|\ge \max_{n}|\be_n^*(x)|$, and for all $k\not \in \supp{(x)}$,  
$$
\|x\|\le \|x+t\be_k\|.
$$
\item\label{shorterconditionb}  For all $x\in \X$, $s, t\in \mathbb{F}$ such that $|t|=|s|\ge \max_{n}|\be_n^*(x)|$, and for all $j<\{k\}\cupdot \supp{(x)}$, 
$$
\|x+s \be_j\|\le \|x+t\be_k\|.
$$
\end{enumerate}
\end{enumerate}
\end{proposition}
\begin{proof}
The equivalence between (i) and (ii) is Remark~\ref{remarkoriginalpslc} when $C_{pl}=1$, whereas the implication from (i) to (iii) is immediate. \\
Suppose now that (iii) holds and let us see that (i) is satisfied. For $A$, $B$, $x$, $t$, $\bfe$ and $\bfe'$ as in (i), we prove by induction on $|B|$ that \eqref{omega_m=1} holds. 
For $|B|=0$, there is nothing to prove, and for $|B|=1$, \eqref{omega_m=1} follows at once from (I) when taking $A=\emptyset$, and  from (II) when taking $|A|=1$. Suppose now that $\eqref{omega_m=1}$ holds for $|B|\le n$. For $|B|=n+1$, take $k$ any element of $B$ and $B_0:=B\setminus \{k\}$.  In the case $A=\emptyset$ applying the inductive hypothesis and then (I) we get
$$
\Vert x \Vert\le \Vert x +t\bff_{\bfe' B_0}\Vert\le \Vert x+t\bff_{\bfe' B_0}+t\bfe' \be_k\Vert= \Vert x+t\bff_{\bfe' B}\Vert. 
$$
In the case $A\not=\emptyset$, define 
$$
j:=\max A, \qquad A_0:=A\setminus \{j\}, \qquad y:=x+t \bfe \be_j, \qquad z:=x+t  \textbf{1}_{\bfe' B_0}.
$$
Since $A_0< \supp(y)\cupdot B_0$, $|A_0|\le |B_0|=n$, and $j<\{k\}\cupdot \supp{(z)}$, applying the inductive hypothesis first and then (II) it follows that 
$$
\Vert x+t\bff_{\bfe A}\Vert=\|y+t \bff_{\bfe A_0}\| \le \Vert y+t\bff_{\bfe' B_0}\Vert= \Vert z+t\bfe \be_j \Vert\le \Vert z+t\bfe' \be_k \Vert=\Vert x+t\bff_{\bfe' B}\Vert.
$$
This completes the induction step, and thus the proof. 
\end{proof}

In \cite{AA1}, the authors ask a difficult - and still open - question
about the relation between $1$-almost greediness and greediness; the central issue is whether $1$-almost greediness implies unconditionality, and thus greediness. A similar question can be asked for strong partial greediness and almost greediness, as follows: \\
\begin{quote}
	If $\mathcal B$ is $1$-strong partially greedy, is $\mathcal B$ $C$-almost greedy for some $C>0$?
\end{quote}

As almost greediness is equivalent to quasi-greediness and democracy, in order to give a negative answer it suffices to show that $1$-strong partial greediness does not imply democracy.
To that end, we use an example from the family given in \cite[Proposition 6.10]{BDKOW}.

\begin{example}
Let $\mathbb X$ be  the completion of $c_{00}$ under the norm
$$
\Vert (a_n)_n\Vert:= \sup_{A\in \mathcal S}\sum_{n\in A}\vert a_n\vert,
$$
where $\mathcal S \subset \mathbb N$ is the set
$$
\mathcal S:=\lbrace A\subset\N\colon \vert A\vert \leq \sqrt{\min A}\rbrace.
$$
The canonical basis $(\be_n)_n$ of $\X$ is 1-PSLC and it is not democratic. 
\end{example}

\begin{proof} 
First, notice that  $(\be_n)_n$ is a normalized 1-unconditional Schauder basis. Thus, Condition~\eqref{shorter}(I) of Lemma~\ref{lemmashorter1PSLC} holds. Now choose $x\in \X$, $s, t\in \mathbb{K}$ so that $|t|=|s|\ge \max_{n}|\be_n^*(x)|$, and $j, k\in \N$ so that $j<\{k\}\cupdot \supp{(x)}$. For $A\in \mathcal{S}$ with $j\not\in A$, we have
$$
\sum_{n\in A}\vert \be_n^*(x+s\be_j)\vert=\sum_{n\in A}\vert \be_n^*(x)\vert\le \sum_{n\in A}\vert \be_n^*(x+t\be_k)\vert\le \|x+t\be_k\|. 
$$
On the other hand, given $A\in \mathcal{S}$ with $j\in A$, let $B:=\{k\}\cup (A\setminus \{j\})$. Since $|B|\le |A|$ and $\min(B)\ge \min (A)$, it follows that $B\in \mathcal{S}$. Hence, 
$$
\sum_{n\in A}\vert \be_n^*(x+s\be_j)\vert=|s|+\sum_{n\in A\setminus \{j\}}\vert \be_n^*(x)\vert=\sum_{n\in B}\vert \be_n^*(x+t\be_k)\vert\le \|x+t\be_k\|.
$$
Taking supremum over all $A\in \mathcal{S}$, we obtain $\|x+s\be_j\| \le \|x+t\be_k\|$. Hence, Condition~\eqref{shorter}(II) of Lemma~\ref{lemmashorter1PSLC} holds as well, and threfore $(\be_n)_n$ has the $1$-PSLC. \\
The fact that the basis is not democratic was proven in \cite{BDKOW}. Here, we include a proof for the sake of completeness. Let $A:=\lbrace m^2+1,\dots,m^2+m\rbrace$ and $B:=\lbrace 1,\dots,m\rbrace$. Then, since $A\in \mathcal S$, $\Vert \mathbf{1}_A\Vert = m$. We claim that $\Vert \mathbf{1}_B\Vert \leq \sqrt{m}$, hence the basis is not democratic. Indeed,  to prove this upper estimate, take a set $A_1 \in \mathcal S$ such that $\Vert \mathbf{1}_B\Vert = \vert A_1\vert$. Then, $\min A_1\leq m$, so $\vert A_1\vert \leq \sqrt{m}$. Thus, $\Vert \mathbf{1}_B\Vert \leq \sqrt{m}$.
\end{proof}

\section{Discussions on the relation between partially greedy, strong partially greedy and quasi-greedy bases.} \label{sectionextensions}
As it has been mentioned in the introduction, in \cite{DKKT} the authors introduced the notion of partially greedy Schauder bases, and characterized them as those which are quasi-greedy and conservative \cite[Theorem 3.4]{DKKT}. Recall that for Schauder bases, being partially greedy and strong partially greedy are equivalent notions. Thus, the following theorem extends \cite[Theorem 3.4]{DKKT} to the context of Markushevich bases, and also shows a relationship between partially and strong partially greedy Markushevich bases.

\begin{theorem}\label{theoremqgpgspg} Let $\mathcal{B}$ be an M-basis in a Banach space $\X$. The following are equivalent.

\begin{enumerate}[\upshape (i)]
\item $\mathcal{B}$ is strong partially greedy. 
\item $\mathcal{B}$ is quasi-greedy and partially greedy. 
 \item  $\mathcal{B}$ is quasi-greedy and superconservative. 
\item  $\mathcal{B}$ is quasi-greedy and conservative.
\end{enumerate}
\end{theorem}
\begin{proof} 
The implications (i) $\Longrightarrow$ (ii) and (iii) $\Longrightarrow$ (iv) are immediate.  To prove that (ii) $\Longrightarrow$ (iii), fix $m\in \N$ and let $A, B, \bfe, \bfe'$ as in Definition~\ref{definitionsuperconservative}. Choose $B_0\subset B$ with $|B_0|=|A|$, and let

$$
m_1:=\max{A}, \qquad m_2:=|B\setminus B_0|, \qquad D:=\{1,\dots,m_1\}\setminus A.
$$
Choose $\eta>0$ and define 
$$
y:=\bff_{\bfe A}+(1+\eta)(\bff_{\bfe' B_0}+\bff_{D}),\qquad  z:=(1+\eta)\bff_{\bfe' B_0}+(1+2\eta)\bff_{\bfe' B\setminus B_0}.
$$
We have
\begin{eqnarray*}
\|\bff_{\bfe A}\|&=&\|y-G_{m_1}(y)\|\le C_p\|y-P_{m_1}(y)\|\\
& = & C_p\|(1+\eta)\bff_{\bfe' B_0}\|=C_p\|z-G_{m_2}(z)\|\\
&\le&C_pC_q\|z\|.
\end{eqnarray*}
Since $\eta$ is arbitrary, it follows that $\|\bff_{\bfe A}\| \le C_pC_q \|\bff_{\bfe' B}\|$ and $\mathbf{sc}_m\le C_pC_q$. Then,  taking supremum over $m$ we obtain $C_{sc}\le C_pC_q$, so the basis is superconservative.\\
To see that (iv) $\Longrightarrow$ (i), note that $\gamma_m\le g_m^c \le C_q$ for all $m\in \N$, so the result follows by Theorem~\ref{one}, Remark~\ref{inequalities} and Proposition~\ref{mejora}. \end{proof}

We do not know whether strong partially and partially greedy Markushevich bases are equivalent notions. In light of Theorem~\ref{theoremqgpgspg}, the question is to find out whether or when every partially greedy Markushevich basis is quasi-greedy. In the case of $1$-partially greedy bases we give a positive answer.

\begin{proposition}\label{proposition1pg} Let $\mathcal B=(\be_n)_{n=1}^\infty$ be $1$-partially greedy Markushevich basis for $\X$ and let $c:=\sup_n\lbrace\Vert \be_n\Vert, \Vert \be_n^*\Vert\rbrace$.
The following hold:
	\begin{enumerate}[\upshape (i)]
		\item  For all $x\in \X$, $m\in \N$, any greedy sum  and $1\le k\le m$, 
		\begin{eqnarray}\label{newpartially-0a} 
		\|x-G_m(x)\|\le \|x-P_k(x)\|.\label{newpartially-0}
		\end{eqnarray}
		\item 
		$\mathcal B$ is $C_{sp}$-strong partially greedy with $C_{sp}\le 1+c^2$.
		\item 
		The sequence $(\be_n)_{n\ge 2}$ is a $1$-strong partially greedy basis for $\Y:=\overline{[\be_n:n\ge 2]}$. 
		\item $\mathcal B$ is $1$-superconservative.
	\end{enumerate}
\end{proposition}
\begin{proof}
We prove \eqref{newpartially-0a} by induction on $m$. For $m=1$, that is just the $1$-partially greedy condition. Suppose \eqref{newpartially-0} holds for $1\le m\le m_0$ and let us prove that it holds for $m_0+1$. Take $x\in \X$ and fix $A$ an $(m_0+1)$-greedy set for $x$ with greedy operator $G_{m_0+1}$. For $k=m_0+1$, \eqref{newpartially-0} holds because $\mathcal{B}$ is $1$-partially greedy. Fix $1\le k_0\le m_0$, and consider first the case that there is $1\le j\le k_0$ such that $j\in A$.
Then
$$
G_{m_0+1}(x)=\be_{j}^*(x)\be_j+G_{m_0}(x-\be_{j}^*(x)\be_j).
$$
Thus, by  inductive  hypothesis,
\begin{eqnarray*}
\|x-G_{m_0+1}(x)\|&=&\|x-\be_{j}^*(x)\be_j-G_{m_0}(x-\be_{j}^*(x)\be_j)\|\\
&\le& \|x-\be_{j}^*(x)\be_j-P_{k_0}(x-\be_{j}^*(x)\be_j)\|=\|x-P_{k_0}(x)\|.
\end{eqnarray*}
On the other hand, if $k_0<A$, then for all $1\le j\le k_0$ and all $1\le l\le m_0+1$, 
	$$
	G_{l}(x)=G_{l}(x-P_j(x));\qquad P_j(x-G_l(x))=P_j(x).
	$$
	Hence, by inductive  hypothesis and the $1$-partially greedy condition, 
	\begin{eqnarray*}
	\|x-G_{m_0+1}(x)\|&=&\|x-G_1(x)-G_{m_0}(x-G_1(x))\|\le \|x-G_1(x)-P_{k_0}(x-G_1(x))\|\\
	&=&\|x-P_{k_0}(x)-G_1(x-P_{k_0}(x))\|\le \|x-P_{k_0}(x)-P_1(x-P_{k_0}(x))\|\\
	&=&\|x-P_{k_0}(x)\|.
	\end{eqnarray*}
	This completes the inductive step, and thus the proof of \eqref{newpartially-0a}.\\
	To prove  (ii), we use  \eqref{newpartially-0a} for $k=1$. For each $x\in \X$ and $m\in \mathbb{N}$, we have
	$$
	\|x-G_m(x)\|\le \|x-P_1(x)\|\le \|x\|+\|P_1(x)\|\le (1+c^2)\|x\|. 
	$$
	Therefore, $\mathcal B$ is strong partially greedy with constant $C_{sp}\le 1+c^2$. \\
	Now, we show (iii). It is clear that $(\be_n)_{n\geq 2}$ is a basis for $\Y$. Denote by $(\overline{P}_m)_m$ and $(\overline{G}_m)_m$ the projections and greedy sums with respect to $(\be_n)_{n\geq 2}$. Given $y\in \Y$, $m\in \N$, and $0\le k\le m$, choose $a>|\be_{{j}}^*(y)|$ for all $j$.  By (i), 
$$
\|y-\overline{G}_m(y)\|=\|a \be_1+y-G_{m+1}(a \be_1+y)\|\le \|a \be_1+y-P_{k+1}(a \be_1+y)\|=\|y-\overline{P}_k(y)\|.
$$
To prove (iv), fix $A\not=\emptyset, B, \bfe, \bfe'$ as in Definition~\ref{definitionsuperconservative}, choose $\eta>0$ and define 
$$
k:=\max{A}, \quad D:=\{1,\dots,k\}\setminus A,\quad y:=\bff_{\bfe A}+(1+\eta)(\bff_{\bfe' B}+\bff_{D}),\quad  m:=|D\cup B|. 
$$
Given that $m\ge k\ge 1$, it follows by (i) that
$$
\|\bff_{\bfe A}\|=\|y-G_m(y)\|\le \|y-P_k(y)\|=(1+\eta)\|\bff_{\bfe' B}\|.
$$
Letting $\eta\rightarrow 0$ we obtain the desired inequality $
\|\bff_{\bfe A}\|\le \|\bff_{\bfe' B}\|$, which shows that $\mathcal B$ is 1-superconservative. Now the proof is complete.
\end{proof}


\begin{thebibliography}{00}
\bibitem{AW} 
\textsc{F. Albiac and P. Wojtaszczyk},
\emph{Characterization of 1-greedy bases}.
J. Approx. Theory 138 (2006), 65--86.

\bibitem{AABW} 
\textsc{F. Albiac and J. L. Ansorena, P. M. Bern\'a, P. Wojtaszczyk},
\emph{Greedy approximation for biorthogonal systems in quasi-Banach spaces} (Submitted, 2019). \url{https://arxiv.org/abs/1903.11651}
		
\bibitem{AA1} 
\textsc{F. Albiac and J. L. Ansorena},
\emph{Characterization of 1-almost greedy bases}.
Rev. Matem. Compl. {\bf 30} (1) (2017), 13--24.
		
\bibitem{BBG} 
\textsc{P. M. Bern\'a, \'O. Blasco, G. Garrig\'os},
\emph{Lebesgue inequalities for the greedy algorithm in general bases}.
Rev. Mat. Complut. \textbf{30} (2017), 369--392.
	
\bibitem{BBGHO} 
\textsc{P. M. Bern\'a, \'O. Blasco, G. Garrig\'os, E. Hern\'andez, T. Oikhberg},
\emph{Embeddings and Lebesgue inequalities for greedy algorithms}.
Constr. Approx. \textbf{48} (3) (2018), 415--451.
	
\bibitem{BBGHO2} 
\textsc{P. M. Bern\'a, \'O. Blasco, G. Garrig\'os, E. Hern\'andez, T. Oikhberg},
\emph{Lebesgue inequalities for Chebyshev Thresholding Greedy Algorithms}.
Rev. Matem. Complut. \url{https://doi.org/10.1007/s13163-019-00328-9}

\bibitem{BDKOW} 
\textsc{P. M. Bern\'a, S. J. Dilworth, D. Kutzarova, T. Oikhberg, B. Wallis},
\emph{The weighted Property (A) and the greedy algorithm}. 
J. Approx. Theory \textbf{248} (2019), 105300. 
	
\bibitem{DKK}
\textsc{S. J. Dilworth, N. J. Kalton, D. Kutzarova}, 
\emph{On the existence of almost greedy bases in Banach spaces},  
Studia Math. 159 (2003), no. 1, 67--101.

\bibitem{DKKT}
\textsc{S. J. Dilworth, N. J. Kalton, D. Kutzarova, and V. N. Temlyakov}, 
\emph{The Thresholding Greedy Algorithm, Greedy Bases, and Duality}, 
Constr. Approx. 19, (2003), 575--597.

\bibitem{DKO2015}
\textsc{S. J. Dilworth,  D. Kutzarova, T. Oikhberg}, 
\emph{Lebesgue constants for the weak greedy algorithm}, 
Rev. Matem. Compl. 28 (2) (2015), 393--409.

\bibitem{DKOSZ}
\textsc{S. J. Dilworth, D. Kutzarova, E. Odell, Th. Schlumprecht and A. Zs\'ak}, \emph{Renorming spaces with greedy bases}. 
J. Approx. Theory, 188 (2014), 39–56.

\bibitem{DST}
\textsc{S. J. Dilworth, M. Soto-Bajo, V. N. Temlyakov},
\emph{Quasi-greedy bases and Lebesgue-type inequalities}. 
Studia Math. {\bf 211} (2012), 41--69.
	
\bibitem{GHO2013}
\textsc{G. Garrig\'os, E. Hern\'andez, T. Oikhberg},
\emph{Lebesgue-type inequalities for quasi-greedy bases}, 
Constr. Approx. 38 (3) (2013), 447--470.
	
\bibitem{KT}
\textsc{S. V. Konyagin and V. N. Temlyakov}, 
\emph{A remark on greedy approximation in Banach spaces}, 
East. J. Approx. 5 (1999), 365--379.


\bibitem{T0}
\textsc{V. N. Temlyakov}, 
Greedy Approximation, Cambridge University Press, 2011.
	
\bibitem{T}
\textsc{V. N. Temlyakov}, 
\emph{Greedy algorithm and n-term trigonometric approximation}, 
Const. Approx. 14 (1998), 569--587.
	
\bibitem{Wo} 
\textsc{P. Wojtaszczyk},
\emph{Greedy Algorithm for General Biorthogonal Systems}, 
J. Approx. Theory, 107, (2000), 293--314.
\end{thebibliography}
\end{document}